\documentclass{article}
\usepackage{amsmath}
\usepackage{amsfonts}
\usepackage{amsthm}
\usepackage{amssymb}
\usepackage{mathrsfs}
\usepackage{hyperref}

\theoremstyle{definition}
\newtheorem{definition}{Definition}

\theoremstyle{remark}

\theoremstyle{plain}
\newtheorem{theorem}{Theorem}
\newtheorem{corollary}{Corollary}
\newtheorem{lemma}{Lemma}
\newtheorem{prop}{Proposition}

\DeclareMathOperator{\im}{Im}
\DeclareMathOperator{\re}{Re}

\DeclareMathOperator{\defeq}{\stackrel{\text{def\textsuperscript{\underline{n}}}}{=}}

\begin{document}

\title{Carleson measures for Hilbert spaces of analytic functions on the complex half-plane}
\author{Andrzej S. Kucik}
\date{}
\maketitle

\begin{flushright}
School of Mathematics \\
University of Leeds \\
Leeds LS2 9JT \\
United Kingdom \\
e-mail: \ttfamily{mmask@leeds.ac.uk}

\end{flushright}

\begin{abstract}
The notion of a \emph{Carleson measure} was introduced by Lennart Carleson in his proof of the Corona Theorem for $H^\infty(\mathbb{D})$. In this paper we will define it for certain type of reproducing kernel Hilbert spaces of analytic functions of the complex half-plane, $\mathbb{C}_+$, which will include Hardy, Bergman and Dirichlet spaces. We will obtain several necessary or sufficient conditions for a positive Borel measure to be Carleson by preforming tests on reproducing kernels, weighted Bergman kernels, and studying the tree model obtained from a decomposition of the complex half-plane. The Dirichlet space will be investigated in detail as a special case. Finally, we will present a control theory application of Carleson measures in determining admissibility of controls in well-posed linear evolution equations.
\end{abstract}

\section{Introduction}

Let $\mu$ be a positive Borel measure on a set $\Omega \subseteq \mathbb{C}$, and let $\mathcal{H}$ be a Hilbert space of complex-valued functions on $\Omega$. If there exists a constant $C(\mu)$, depending on $\mu$ only, such that for all $h \in \mathcal{H}$ we have
\begin{equation}
	\int_{\Omega} |h|^2 \, d\mu \leq C(\mu) \|h\|^2_{\mathcal{H}},
	\label{eq:carlesonembedding}
\end{equation}
then $\mu$ is called a \emph{Carleson measure} for $\mathcal{H}$ and we shall refer to \eqref{eq:carlesonembedding} as the \emph{Carleson criterion}. The set of Carleson measures for $\mathcal{H}$ will be denoted by $CM(\mathcal{H})$. The notion of a Carleson measure was introduced by Lennart Carleson in his proof of the Corona Theorem for $H^\infty(\mathbb{D})$ in \cite{carleson1962}, where a complete characterisation of Carleson measures for $H^p(\mathbb{D}), (1 \leq p < \infty)$ was given. In 1967 Lars H\"{o}rmander extended Carleson's result to the unit ball of $\mathbb{C}^n$ \cite{hormander1967}, and since then many other generalisations and variants of this idea have been studied (we mention in particular the characterisation of Carleson measures for the weighted Bergman spaces on $\mathbb{D}$ by J. Cima and W. Wogen in \cite{cima1982} and on the unit ball of $\mathbb{C}^n$ by D. Luecking in \cite{luecking1983}, and for the weighted Dirichlet space on $\mathbb{D}$ by D. Stegenga in \cite{stegenga1980}). The popularity of this area of research is a consequence of wide applicability of Carleson embeddings, going far beyond Carleson's original formulation of this concept, and in particular their usefulness in the study of certain classes of operators acting on $\mathcal{H}$ (for example multiplication operators \cite{kucik2015}, \cite{stegenga1980}). However, this area of research is usually limited to the case of $\Omega = \mathbb{D}$ or the unit ball of $\mathbb{C}^n$, and other domains are rarely considered. 

In this paper we shall consider 
\begin{displaymath}
	\Omega = \mathbb{C}_+ := \left\{z =x+iy \in \mathbb{C} \; : \; x > 0 \right\},
\end{displaymath}
the open right complex half-plane. This choice of domain is not arbitrary and its motivation is drawn from two main reasons. First of all, for some of the most well-known Hilbert spaces of analytic functions on the open unit disk of the complex place, such as the Hardy space $H^2$ \cite{duren1970}, \cite{mashreghi2009}, the Bergman space $\mathcal{B}^2$ \cite{duren2004}, \cite{hedenmalm2000} or the Dirichlet space $\mathcal{D}$ \cite{arcozzi2011}, \cite{el-fallah2014}, there exists a fundamental relation between the norm on each of this spaces and the norm of some weighted sequence space $\ell^2$, namely
\begin{align*}
\|f\|^2_{H^2} &:= \sup_{0 < r < 1} \frac{1}{2\pi} \int_0^{2\pi} \left|f(re^{i\theta})\right|^2 \, d\theta = \sum_{n=0}^\infty |a_n|^2 &  \left(\forall f(z) = \sum_{n=0}^\infty a_n z^n \in H^2 \right), \\
\|f\|^2_{\mathcal{B}^2} &:=  \frac{1}{\pi} \int_{\mathbb{D}}\left|f(z)\right|^2 \, dz = \sum_{n=0}^\infty \frac{|a_n|^2}{n+1} &\left(\forall f(z) = \sum_{n=0}^\infty a_n z^n \in \mathcal{B}^2 \right), \\
\|f\|^2_{\mathcal{D}} &:=  \|f\|^2_{H^2}+\|f'\|^2_{\mathcal{B}^2} = \sum_{n=0}^\infty (n+1)|a_n|^2 &  \left(\forall f(z) = \sum_{n=0}^\infty a_n z^n \in \mathcal{D}\right).
\end{align*}
But of course for some problems it is more natural to consider the continuous version of the weighted sequence space $\ell^2$, that is the weighted $L^2(0, \, \infty)$ space. It follows from the Plancherel's Theorem, that for some class of weights, the Laplace transform ($\mathfrak{L}$) is an isometric map from the weighted space of square-(Lebesgue)-integrable functions on the positive real half-line to some spaces of analytic functions defined on the open right complex half-plane (we shall present this statement rigorously in the next section). And for example, if we denote by $H^2(\mathbb{C}_+), \, \mathcal{B}^2(\mathbb{C}_+)$ and $\mathcal{D}(\mathbb{C}_+)$ the spaces of Hardy, Bergman and Dirichlet (respectively) on the half-plane, we have:
\begin{align*}
\|F\|^2_{H^2(\mathbb{C}_+)} &:= \sup_{x>0} \int_{-\infty}^\infty \left|F(x+iy)\right|^2 \, \frac{dy}{2\pi} = \int_0^\infty |f(t)|^2 \, dt, \\
\|F\|^2_{\mathcal{B}^2(\mathbb{C}_+)} &:= \int_{\mathbb{C}_+} \left|F(z)\right|^2 \, \frac{dz}{\pi} = \int_0^\infty |f(t)|^2 \, \frac{dt}{t}, \\
\|F\|^2_{\mathcal{D}(\mathbb{C}_+)} &:= \|F\|^2_{H^2(\mathbb{C}_+)}+\|F'\|^2_{\mathcal{B}^2(\mathbb{C}_+)}=\int_0^\infty |f(t)|^2 (t+1) \, dt,
\end{align*}
for all $F=\mathfrak{L}[f]$ in $H^2(\mathbb{C}_+)$, or in $\mathcal{B}^2(\mathbb{C}_+)$, or in $\mathcal{D}(\mathbb{C}_+)$ and $f$ in an appropriate weighted $L^2$ space on $(0, \, \infty)$.

One of the instances where the continuous setting is more suitable, and also the second reason motivating our research of Carleson measures for these spaces, is the study of control and observation operators for linear evolution equations. It has been shown in \cite{jacob2014} that the admissibility criterion for these operators is equivalent to the Carleson criterion. We shall explain it in the concluding section of this paper.

In Section \ref{sec:preliminaries} we will introduce the definition of spaces which will be studied in this paper and present their relation to the weighted $L^2$ spaces on $(0, \, \infty)$ via the isometric map defined by the Laplace transform. In Section \ref{sec:kernels} we  will perform some tests on reproducing and weighted Bergman kernels in order to obtain sufficient as well as necessary conditions to satisfy the Carleson criterion. Carleson measures for the Dirichlet space will be characterised in Section \ref{sec:dirichlet}. Following that, in Section \ref{sec:trees}, we will introduce some techniques of analysis on trees to produce a sufficient condition for a measure to be Carleson for spaces which are generalisations of the Dirichlet space. Finally, an application of these results to control theory will be given in Section \ref{sec:application}.

\section{Preliminaries}
\label{sec:preliminaries}
Let us now present some essential definitions and results. Let $\tilde{\nu}$ be a positive regular Borel measure on $[0, \, \infty)$ satisfying the so-called \emph{$\Delta_2$-condition}:
\begin{equation}
	\sup_{r>0} \frac{\tilde{\nu}[0, \, 2r)}{\tilde{\nu}[0, \, r)}  < \infty,
	\tag{$\Delta_2$}
	\label{eq:delta2}
\end{equation}
and let $\lambda$ denote the Lebesgue measure on $i\mathbb{R}$. We define $\nu := \tilde{\nu} \otimes \lambda$ to be a positive regular Borel measure on the closed right complex half-plane $\overline{\mathbb{C}_+}~:=~[0, \, \infty) \times i\mathbb{R}$. For this measure and $1 \leq p < \infty$ a \emph{Zen space} \cite{jacob2013} is defined to be:
\begin{displaymath}
	A^p_\nu := \left\{F : \mathbb{C}_+ \longrightarrow \mathbb{C} \, \text{analytic} \; : \; \left\|F\right\|^p_{A^p_\nu} := \sup_{\varepsilon >0} \int_{\mathbb{C}_+} \left|F(z+\varepsilon)\right|^p \, d\nu < \infty \right\}.
\end{displaymath}
The Zen space definition naturally extends the definition of weighted Bergman spaces, $\mathcal{B}^p_\alpha(\mathbb{C}_+)$. Indeed, if $d\tilde{\nu}(r)=r^\alpha dr/\pi$, for some $\alpha>-1$, then $A^p_\nu~=~\mathcal{B}^p_\alpha(\mathbb{C}_+)$ (the fact that both Zen and Bergman spaces are usually denoted by the letter $A$ justifies why we chose to label the latter with $\mathcal{B}$, avoiding potential confusions). Also, if $2\pi\tilde{\nu}=\delta_0$, the Dirac measure with point mass at 0, then $A^p_\nu= H^p(\mathbb{C}_+)$, which we may also identify with $\mathcal{B}^2_{-1}(\mathbb{C}_+)$. 
If we now assume that $(\nu_n)_{n=0}^m~=~(\tilde{\nu}_n \otimes \lambda)_{n=0}^m, \, m \in \mathbb{N} \cup \{\infty\}$, is a sequence of positive regular Borel measures on $\overline{\mathbb{C}_+}$, each of which satisfies \eqref{eq:delta2}-condition, we can define a new space of functions, further extending the definition of Zen spaces (and consequently weighted Bergman spaces), by setting
\begin{displaymath}
	A^p\left( \mathbb{C}_+, \, (\nu_n)_{n=0}^m \right) := \left\{F : \mathbb{C}_+ \longrightarrow \mathbb{C} \, \text{analytic} \; : \; \left\|F\right\|^p_{A^p_\nu} := \sum_{n=0}^m \left\|F^{(n)}\right\|^p_{A^p_{\nu_n}} < \infty \right\}.
\end{displaymath}
In case of $p=2$, the relation between these spaces and the weighted $L^2$ spaces on $(0, \, \infty)$ has been introduced in \cite{kucik2015} and we will quote some of the results here.
\begin{theorem}
\label{thm:mainthm}
The Laplace transform defines an isometric map 
\begin{displaymath}
	\mathfrak{L} : L^2_{w_{(m)}} (0, \, \infty) \longrightarrow A^2\left(\mathbb{C}_+, \, (\nu_n)_{n=0}^m\right),
\end{displaymath}
where
\begin{equation}
	w_{(m)} := \sum_{n=0}^m w_n(t) \qquad \text{ and } \qquad w_n(t) := 2\pi t^{2n}\int_0^\infty e^{-2rt} \, d\tilde{\nu}_n(r) \quad (t > 0).
\label{eq:wm}
\end{equation}
Here by $L^2_{w_{(m)}}(0, \, \infty)$ we mean the Hilbert space of functions $f : (0, \, \infty) \longrightarrow \mathbb{C}$ such that
\begin{displaymath}
	\|f\|^2_{L^2_{w_{(m)}}(0, \, \infty)} := \int_0^\infty |f(t)|^2 w_{(m)}(t) \, dt < \infty.
\end{displaymath}
\end{theorem}
For $m=0$, this result was proved in \cite{jacob2013}, with some partial results appearing earlier in \cite{zen2009}, \cite{zen2010}, and also in \cite{das1994}, \cite{duren2007}. Allowing $m>0$ enables us to consider $L^2$ spaces with non-decreasing weights, such as $w_{(1)}(t)=1+t$, which by the virtue of the above theorem, applied to $\tilde{\nu}_0 = \delta_0/2\pi$ and $\tilde{\nu}_1$ being the weighted (with weight $1/\pi$) Lebesgue measure on $[0, \, \infty)$, corresponds to the Dirichlet space on $\mathbb{C}_+$.

If the choice of $(\nu_n)_{n=0}^m$ is implicit and unambiguous, we shall adopt the notation
\begin{displaymath}
	A^2_{(m)} = \mathfrak{L}\left(L^2_{w_{(m)}}(0, \, \infty) \right).
\end{displaymath}
This is a reproducing kernel Hilbert space (RKHS), with the kernel given by
\begin{equation}
k^{A^2_{(m)}}_z(\zeta) = \int_0^\infty \frac{e^{-t(\overline{z}+\zeta)}}{w_{(m)}(t)} \, dt \qquad \qquad (\forall (z, \zeta) \in \mathbb{C}_+),
\label{eq:kernel}
\end{equation}
for details see again \cite{kucik2015}.

\section{Kernel conditions}
\label{sec:kernels}
Since the space $A^2_{(m)}$ is a generalisation of the Dirichlet space, some of the methods used to characterise the Carleson measures for the latter space, can also be employed here. The classical Dirichlet space, $\mathcal{D}$, is defined on the open unit disk of the complex plane (with some obvious extensions to the $n$-dimensional case). The Carleson measures for $\mathcal{D}$ have been completely classified by D. Stegenga in \cite{stegenga1980}, using the notion of so-called logarithmic capacity. A number of other characterisations has been obtained later, and although none of them is particularly simple, we feel obliged to at least mention a paper \cite{arcozzi2002} by Arcozzi, Rochberg and Sawyer since some of the results given there have their half-plane counterparts which are proven in Section \ref{sec:trees} of this article. Many of these characterisations however rely heavily on the fact that $\mathbb{D}$ is bounded, and the Dirichlet spaces defined on unbounded domains are virtually never considered. For example, Stegenga's logarithmic capacity classification of Carleson measures is altogether unsuitable. But yet, some weaker results may be adopted to the $\mathbb{C}_+$ case. Let us consider the following adaptations of Theorem 5.2.2 (p. 76) from \cite{el-fallah2014}.

\begin{lemma}
Let $\mu$ be a positive Borel measure on $\mathbb{C}_+$, then
\small
\begin{equation}
	\sup_{\left\|F\right\|_{A^2_{(m)}} \leq 1} \int_{\mathbb{C}_+} |F(z)|^2 \, d\mu(z) = \sup_{\|G\|_{L^2(\mathbb{C}_+, \, \mu)} \leq 1} \left|\int_{\mathbb{C}_+} \int_{\mathbb{C}_+} k^{A^2_{(m)}}_z (\zeta) \, G(\zeta) \overline{G(z)} \, d\mu(z) \, d\mu(\zeta) \right|.
\label{eq:ransford}
\end{equation}
\normalsize
\end{lemma}

\begin{proof}
If the LHS of \eqref{eq:ransford} is finite (i.e. $\mu$ is a Carleson measure for $A^2_{(m)}$), then the proof is essentially the same as the proof of Theorem 5.2.2 from \cite{el-fallah2014}. If it is not finite, let
\begin{displaymath}
	\Omega_r = \left\{x+iy \in \mathbb{C}_+ \; : \; \frac{1}{r} \leq x \leq r, \; -r \leq y \leq r \right\} \subset \mathbb{C}_+ \; \; \; \; \; (r > 0).
\end{displaymath}
Then
\begin{displaymath}
\int_{\mathbb{C}_+} |F|^2 \, d\mu|_{\Omega_r} \leq  \mu(\Omega_r) \sup_{z \in \Omega_r} \left\|k_z^{A^2_{(m)}}\right\|^2 \left\|F\right\|^2_{A^2_{(m)}} \; \; \; \; \; (F \in A^2_{(m)}),
\end{displaymath}
and hence $\mu|_{\Omega_r}$ (i.e. the restriction of $\mu$ to $\Omega_r$) is a Carleson measure for $A^2_{(m)}$, so we can use the first part of the proof, that is we are given that \eqref{eq:ransford}, to get
\begin{align*}
		\sup_{\left\|F\right\|_{A^2_{(m)}} \leq 1} \int_{\mathbb{C}_+} &|F(z)|^2 \, d\mu|_{\Omega_r}(z) \\
		= &\sup_{\|G\|_{L^2(\mathbb{C}_+, \, \mu)} \leq 1} \left|\int_{\mathbb{C}_+} \int_{\mathbb{C}_+} k^{A^2_{(m)}}_z (\zeta) \, G(\zeta) \overline{G(z)} \, d\mu|_{\Omega_r}(z) \, d\mu|_{\Omega_r}(\zeta) \right|,
\end{align*}
where the RHS is at most equal to the RHS of \eqref{eq:ransford} and the LHS tends to infinity as $r$ approaches infinity.
\end{proof}

\begin{prop}
\label{ransfordprop}
If
\begin{equation}
	\sup_{z \in \mathbb{C}_+} \int_{\mathbb{C}_+} \left|k_z^{A^2_{(m)}}\right| \, d\mu < \infty,
	\label{eq:weakercondition}
\end{equation}
then $\mu$ is a Carleson measure for $A^2_{(m)}$.
\end{prop}

\begin{proof}
Let
\begin{equation}
	M := \sup_{z \in \mathbb{C}_+} \int_{\mathbb{C}_+} \left|k_z^{A^2_{(m)}}\right| \, d\mu.
\label{eq:M}
\end{equation}
Then for all $G \in L^2(\mathbb{C}_+, \, \mu)$
\begin{align}
\begin{split}
	&\left|\int_{\mathbb{C}_+} \int_{\mathbb{C}_+} k^{A^2_{(m)}}_z (\zeta) \, G(\zeta) \overline{G(z)} \, d\mu(z) \, d\mu(\zeta) \right| \\
	&\stackrel{\text{H\"{o}lder's}}{\leq} \left(\int_{\mathbb{C}_+} \int_{\mathbb{C}_+} \left|k^{A^2_{(m)}}_\zeta(z)\right| \, |G(\zeta)|^2 \, d\mu(z) \, d\mu(\zeta) \right)^{1/2} \\
	&\qquad \; \times \left(\int_{\mathbb{C}_+} \int_{\mathbb{C}_+} \left|k^{A^2_{(m)}}_z(\zeta)\right| \, |G(z)|^2 \, d\mu(z) \, d\mu(\zeta) \right)^{1/2} \\
	& \quad \stackrel{\eqref{eq:M}}{\leq} \quad M \|G\|^2_{L^2(\mathbb{C}_+, \, \mu)}.
	\end{split}
	\label{eq:MG}
\end{align}
Therefore
\begin{align*}
\int_{\mathbb{C}_+} \left(\frac{|H(z)|}{\left\|H\right\|_{A^2_{(m)}}}\right)^2 \, &d\mu(z) \\
&\leq \sup_{\left\|F\right\|_{A^2_{(m)}} \leq 1} \int_{\mathbb{C}_+} |F(z)|^2 \, d\mu(z) \\
&\stackrel{\eqref{eq:ransford}}{=} \sup_{\|G\|_{L^2(\mathbb{C}_+, \, \mu)} \leq 1} \left|\int_{\mathbb{C}_+} \int_{\mathbb{C}_+} k^{A^2_{(m)}}_z (\zeta) \, G(\zeta) \overline{G(z)} \, d\mu(z) \, d\mu(\zeta) \right| \\
&\stackrel{\eqref{eq:MG}}{\leq} M,
\end{align*}
for all $H \in A^2_{(m)}$, as required.
\end{proof}

In the RKHS case, in order to obtain necessary conditions for a measure to be Carleson, it is also a fairly standard practice to test Carleson criterion on reproducing kernels. However, in $A^2_{(m)}$ it often brings rather disappointing results, as the reproducing kernels of $A^2_{(m)}$ are seldom expressible as elementary functions (recall formulae \eqref{eq:wm} and \eqref{eq:kernel}). This can be overcome if the r\^{o}le of reproducing kernels of $A^2_{(m)}$ is assumed by the reproducing kernels of weighted Bergman spaces. Recall that the weighted Bergman space on the right complex half-plane, $\mathcal{B}^p_{\alpha}(\mathbb{C}_+)$, $\alpha~\geq~-1$, is the Banach space of analytic functions $F:\mathbb{C}_+ \longrightarrow \mathbb{C}$, such that
\begin{displaymath}
	\|F\|^p_{\mathcal{B}^p_{\alpha}} := \frac{1}{\pi} \int_{-\infty}^\infty \int_0^\infty |F(x+iy)|^p x^\alpha\, dx dy < \infty \qquad \qquad (\alpha > -1),
\end{displaymath}
and $\mathcal{B}^p_{-1}(\mathbb{C}_+) := H^p(\mathbb{C}_+)$, \cite{elliott2011}. If $p=2$, then $\mathcal{B}^2_{\alpha}(\mathbb{C}_+)$ is a RKHS, and
\begin{displaymath}
	k^{\mathcal{B}^2_{-1}(\mathbb{C}_+)}_z(\zeta) \defeq k^{H^2(\mathbb{C}_+)}_\zeta(z) \stackrel{\eqref{eq:kernel}}{=} \frac{1}{\overline{z}+\zeta} \quad \text{and} \quad k^{\mathcal{B}^2_{\alpha}(\mathbb{C}_+)}_z(\zeta) \stackrel{\eqref{eq:kernel}}{=} \frac{2^\alpha(1+\alpha)}{(\overline{z}+\zeta)^{2+\alpha}}, \; \alpha>-1,
\end{displaymath}
for all $(z, \, \zeta) \in \mathbb{C}_+^2$. We shall call all the functions of the form 
\begin{displaymath}
	K_\alpha(z, \, \zeta):=(\overline{z}+\zeta)^{-2-\alpha}, \qquad \qquad (z, \, \zeta) \in \mathbb{C}_+^2, \quad \alpha \geq -1,
\end{displaymath}
the \emph{Bergman kernels for the right complex half-plane}.

\begin{lemma}
Suppose that $m \in \mathbb{N}_0$. Then there exists $\alpha_0 \geq -1$ such that for all $(z, \, \zeta) \in \mathbb{C}_+^2$ and $\alpha \geq \alpha_0$, $K_\alpha(z, \zeta)$ is in $A^2_{(m)}$ (viewed as an analytic function in $\zeta$).
\label{bergmanlemma}
\end{lemma}

\begin{proof}
For each $0 \leq n \leq m$, let
\begin{equation}
R_n : = \sup_{r>0} \frac{\tilde{\nu}_n[0, 2r)}{\tilde{\nu}_n[0, r)}.
\label{eq:Rn}
\end{equation}
Clearly, for all $r>0$
\begin{displaymath}
\tilde{\nu}_n[0, r) + \tilde{\nu}_n[r, 2r) = \tilde{\nu}_n[0, 2r) \stackrel{\eqref{eq:Rn}}\leq R_n \tilde{\nu}_n[0, r),
\end{displaymath}
so
\begin{equation}
\tilde{\nu}_n[r, 2r) \leq (R_n-1) \tilde{\nu}_n[0, r).
\label{eq:Rncondition}
\end{equation}
Choose $q >0$ such that
\begin{displaymath}
	2^q > \sup_{0 \leq n \leq m} R_n. 
\end{displaymath}
Define $g:[0, \infty) \longrightarrow (0, \infty)$ to be a step function such that $g(r) = \re(z)^{-q}$, if $0 \leq r<1$, and $g(r) = (2^j+\re(z))^{-q}$, if $r \in [2^j, \, 2^{j+1})$, for all $j \in \mathbb{N}_0$.
\begin{align*}
\int_0^\infty g(r) \, d\tilde{\nu}_n(r) &= \frac{\tilde{\nu}_n[0,1)}{\re(z)^q} + \sum_{j=0}^\infty \int_{[2^j, \, 2^{j+1})} \frac{d\tilde{\nu}_n(r)}{(2^j+\re(z))^q} \\
&\defeq \frac{\tilde{\nu}_n[0, 1)}{\re(z)^q} + \sum_{j=0}^\infty \frac{\tilde{\nu}_n[2^j, 2^{j+1})}{(2^j+\re(z))^q} \\
&\stackrel{\eqref{eq:Rncondition}}{\leq} \frac{\tilde{\nu}_n[0, 1)}{\re(z)^q} + (R_n-1) \sum_{j=0}^\infty \frac{\tilde{\nu}_n[0, 2^j)}{(2^j+\re(z))^q} \\
&\stackrel{\eqref{eq:Rn}}{\leq} \tilde{\nu}_n[0, 1) \left( \frac{1}{\re(z)^q}+ (R_n-1) \sum_{j=0}^\infty \left(\frac{R_n}{2^q}\right)^j \right) < \infty,
\end{align*}
for all $0 \leq n \leq m$. It follows that
\begin{align*}
\int_0^\infty \left|t^{\alpha+1} e^{-t\overline{z}} \right|^2 w_n(t) \, dt &= 2\pi \int_0^\infty \int_0^\infty t^{2(\alpha+n+1)} e^{-2t(r+\re(z))} \, dt \, d\tilde{\nu}_n(r) \\
 &= \frac{\pi \Gamma(2\alpha+2n+3)}{2^{2\alpha+2n+2}} \int_0^\infty \frac{d\tilde{\nu}_n(r)}{(r+\re(z))^{2\alpha+2n+3}}  \\
&\leq \frac{\pi \Gamma(2\alpha+2n+3)}{2^{2\alpha+2n+2}}\int_0^\infty g(r) \, d\tilde{\nu}_n(r) < \infty,
\end{align*}
whenever $\alpha \geq \alpha_0 := (q-3)/2$. And consequently, by Theorem \ref{thm:mainthm}, we have
\begin{displaymath}
\mathfrak{L}\left[ \frac{t^{\alpha+1} e^{-t\overline{z}}}{\Gamma(\alpha+2)} \right](\zeta) = \frac{1}{(\overline{z}+\zeta)^{\alpha+2}} = K_\alpha(z, \, \zeta)\; \in A^2_{(m)}.
\end{displaymath}
\end{proof}

\begin{definition}
Let $a \in \mathbb{C}_+$. A \emph{Carleson square centred at $a$} is defined to be the subset
\begin{equation}
Q(a) := \left\{z = x+iy \; : \; 0<x\leq 2\re(a), \, |y-\im(a)| \leq \re(a)\right\}
\label{eq:carlesonsquare}
\end{equation}
of the open right complex half-plane.
\end{definition}

\begin{theorem}
Suppose that $m \in \mathbb{N}_0$. If $\mu$ is a Carleson measure for the space $A^2(\mathbb{C}_+, \, (\nu_n)_{n=0}^m)$, then there exists a constant $C(\mu)>0$ such that
\begin{equation}
	\mu(Q(a)) \leq C(\mu) \sum_{n=0}^m \frac{\nu_n\left(\overline{Q(a)}\right)}{\re(a)^{2n}},
	\label{eq:sufficientcarlesonforA2}
\end{equation}
for all $a \in \mathbb{C}_+$. Here $\overline{Q(a)}$ denotes the closure of $Q(a)$ in $\mathbb{C}$.
\label{carlesonforA2}
\end{theorem}

\begin{proof}
Given $a \in \mathbb{C}_+$, we have that for each $z \in Q(a)$,
\begin{equation}
|z+\overline{a}| \defeq \sqrt{(\re(z)+\re(a))^2+(\im(z)-\im(a)^2)} \stackrel{\eqref{eq:carlesonsquare}}{\leq} \sqrt{10}\re(a).
\label{eq:carlesonsq}
\end{equation}
We also know, by Lemma \ref{bergmanlemma}, that there exists $\beta \geq 0$ such that \\ $t^\beta e^{-t\overline{a}} \in L^2_{w_{(m)}}(0, \, \infty)$, and
\begin{equation}
	\left\|\mathfrak{L}[t^\beta e^{-t\overline{a}}]\right\|^2_{L^2(\mathbb{C}_+, \, \mu)} \geq (\Gamma(\beta+1))^2 \int_{Q(a)} \frac{d\mu(z)}{|z+\overline{a}|^{2(\beta+1)}} \stackrel{\eqref{eq:carlesonsq}}{\geq} \frac{(\Gamma(\beta+1))^2\mu(Q(a))}{10^{\beta+1}\re(a)^{2(\beta+1)}}.
\label{eq:lowercarlesonestimate}
\end{equation}
On the other hand
\begin{align*}
\left\|\mathfrak{L}[t^\beta e^{-t\overline{a}}]\right\|^2_{A^2(\mathbb{C}_+, \, (\nu_n)_{n=0}^m)} &\stackrel{\text{Th\textsuperscript{\underline{m}} }\ref{thm:mainthm}}{=} \left\|t^\beta e^{-t\overline{a}}\right\|^2_{L^2_{w_{(m)}(0, \infty)}} = \int_0^\infty t^{2\beta}e^{-2\re(a)t} w_{(m)}(t) \, dt\\
&\stackrel{\eqref{eq:wm}}{=} 2\pi \sum_{n=0}^m \int_0^\infty \int_0^\infty t^{2(\beta+n)} e^{-2(r+\re(a))t} \, dt \, d\tilde{\nu}_n(r) \\
& = 2\pi \sum_{n=0}^m \frac{\Gamma(2\beta+2n+1)}{2^{2\beta+2n+1}} \int_0^\infty \frac{d\tilde{\nu}_n(r)}{(r+\re(a))^{2\beta+2n+1}}.
\end{align*}
And again, letting $R_n$ be defined like in \eqref{eq:Rn}, for all $0\leq n \leq m$, and using essentially the same method as that in the proof of Lemma \ref{bergmanlemma} (with $\tilde{\nu}_n[0, \re(a))$ instead of $\tilde{\nu}_n[0, 1)$, for each $0\leq n\leq m$), we get that the last expression is less or equal to:
\begin{displaymath}
2\pi \sum_{n=0}^m \frac{\Gamma(2\beta+2n+1)\tilde{\nu}_n[0, \re(a))}{(2\re(a))^{2\beta+2n+1}}  \left( 1 + (R_n-1) \sum_{j=0}^\infty \frac{R_n^j}{(1+2^j)^{2\beta+2n+1}} \right),
\end{displaymath}
and the series converges for $\beta$ sufficiently large. Therefore, combining this with \eqref{eq:lowercarlesonestimate}, we get
\begin{displaymath}
	\mu \left(Q(a)\right) \leq C(\mu) \sum_{n=0}^m \frac{2\re(a)^{2(\beta+1)}\tilde{\nu}_n[0, \re(a))}{\re(a)^{2\beta+2n+1}} \leq C(\mu) \sum_{n=0}^m \frac{\nu_n\left(\overline{Q(a)}\right)}{\re(a)^{2n}}, 
\end{displaymath}
where
\small
\begin{displaymath}
	C(\mu) := \frac{10^{\beta+1}\pi\Gamma(2\beta+2m+1)}{2^{2\beta-1}(\Gamma(\beta+1))^2} \left[ 1 + \max_{0\leq n \leq m}\left((R_n-1) \sum_{j=0}^\infty \frac{R_n^j}{(1+2^j)^{2\beta+2n+1}} \right)\right]C,
\end{displaymath}
\normalsize
and $C>0$ is a Carleson constant from the embedding
\begin{displaymath}
	A^2(\mathbb{C}_+, \, (\nu_n)_{n=0}^m)~\hookrightarrow~L^2(\mathbb{C}_+, \, \mu).
\end{displaymath}
\end{proof}

In \cite{jacob2013} the condition \eqref{eq:sufficientcarlesonforA2} was proved to be equivalent to the Carleson criterion, if $m=0$ (i.e. for Zen spaces). It is not clear if this remains true for $m>1$.

\section{The Dirichlet space(s)}
\label{sec:dirichlet}

Let us now consider a particularly well-known example of $A^2_{(1)}$, corresponding to measures $\tilde{\nu}_0 = \delta_0/2\pi$ and $\tilde{\nu}_1$ being the weighted Lebesgue measure on $[0, \, \infty)$ (with weight $1/\pi$), or alternatively to the Laplace image of $L^2_{1+t}(0, \, \infty)$. That is the Dirichlet space $\mathcal{D}(\mathbb{C}_+)$. The previous section provides us with some information about the set of Carleson measures for $\mathcal{D}(\mathbb{C}_+)$.

\begin{prop}
Let $\mu$ be a positive Borel measure on $\mathbb{C}_+$. Then
\begin{enumerate}
	\item If for each $a \in \mathbb{C}_+$
	\begin{equation}
		\mu(Q(a)) = O(\re(a)),
		\label{eq:carleson1}
	\end{equation}
	then $\mu$ is a Carleson measure for  $\mathcal{D}(\mathbb{C}_+)$.
	\item If $\mu$ is a Carleson measure for  $\mathcal{D}(\mathbb{C}_+)$, then
	\begin{displaymath}
		\mu(Q(a)) = O(\re(a)+1)
	\end{displaymath}
	for each $a \in \mathbb{C}_+$.
\end{enumerate}
\end{prop}

\begin{proof}
If \eqref{eq:carleson1} holds, then $\mu$ is a Carleson measure for the Hardy space $H^2(\mathbb{C}_+)$, so it must also be a Carleson measure for $\mathcal{D}(\mathbb{C}_+)$. Conversely, if it is a Carleson measure for $\mathcal{D}(\mathbb{C}_+)$, then, by the previous theorem, there exists a constant $C(\mu)>0$ such that
\begin{displaymath}
	\mu(Q(a)) \stackrel{\eqref{eq:sufficientcarlesonforA2}}{\leq} C(\mu) \left(\nu_0\left(\overline{Q(a)}\right) + \frac{\nu_1\left(\overline{Q(a)}\right)}{\re(a)^2}\right) \leq 2C(\mu) \left(\re(a) + 2 \right).
\end{displaymath}
\end{proof}

On the open unit disk of the complex plane the Dirichlet space, $\mathcal{D}$, is defined to be the Banach space of analytic functions with derivatives in the (unweighted) Bergman space, $\mathcal{B}$. The quantity
\begin{equation}
\int_{\mathbb{D}} |f'(z)|^2 \, dz \qquad \qquad \qquad (f \in \mathcal{D})
\label{eq:dirichletseminorm}
\end{equation}
is a seminorm on $\mathcal{D}$. A norm on $\mathcal{D}$ can be defined by adding an absolute value of the evaluation of $f$ at a constant to \eqref{eq:dirichletseminorm} or by adding to it the Hardy norm, $\|\cdot\|_2$, (it is always possible, since $\mathcal{D} \subset H^2$). On the disk these two norms are equivalent \cite{arcozzi2011}, \cite{el-fallah2014}. On the complex half-plane, however, it not the case.

Let us consider the following variant of the Dirichlet space on the right complex half-plane: given $\alpha \in \mathbb{C}_+$, let
\begin{displaymath}
	\mathcal{D}^\alpha(\mathbb{C}_+) := \left\{F : \mathbb{C}_+ \longrightarrow \mathbb{C} \; \text{analytic} \; : \; \|F'\|^2_{\mathcal{B}^2_0(\mathbb{C}_+)} \defeq \int_{\mathbb{C}_+} |F'(z)|^2 \, \frac{dz}{\pi} < \infty \right\},
\end{displaymath}
with the inner product given by
\begin{displaymath}
	\left\langle F, \, G \right\rangle_{\mathcal{D}^\alpha(\mathbb{C}_+)} := \left\langle F', \, G' \right\rangle_{\mathcal{B}^2_0(\mathbb{C}_+)} + F(\alpha)\overline{G(\alpha)}.
\end{displaymath}
It is a reproducing kernel Hilbert space and we can find its reproducing kernel, $k_z^{\mathcal{D}^\alpha(\mathbb{C}_+)}$, using the following method. Let $F \in \mathcal{D}^\alpha(\mathbb{C}_+)$ be such that $F' = \mathfrak{L}'[f]$ for some $f \in L^2_{1/t}(0, \, \infty)$. Then
\begin{displaymath}
F'(z) = \left\langle F', \, k_z^{\mathcal{B}^2_0(\mathbb{C}_+)} \right\rangle_{\mathcal{B}^2_0(\mathbb{C}_+)} = \frac{d}{dz} \left\langle F, \, k_z^{\mathcal{D}^\alpha(\mathbb{C}_+)} \right\rangle_{\mathcal{D}^\alpha(\mathbb{C}_+)}.
\end{displaymath}
So by the Fundamental Theorem of Calculus,
\begin{align*}
	F(z) &= \int_{\mathbb{C}_+} F'(\zeta) \int_\alpha^z \frac{d\xi}{\pi(\xi +\overline{\zeta})^2} \, d\zeta + F(\alpha) \\
	&= \left\langle F', \, \left(k_z^{\mathcal{D}^\alpha(\mathbb{C}_+)}\right)' \right\rangle_{\mathcal{B}^2_0(\mathbb{C}_+)} + F(\alpha)k_z^{\mathcal{D}^\alpha(\mathbb{C}_+)}(\alpha).
\end{align*}
And by the uniqueness property of reproducing kernels (\cite{aronszajn1950}, \cite{paulsen2009}),
\begin{displaymath}
	\left(k_z^{\mathcal{D}^\alpha(\mathbb{C}_+)}\right)'(\zeta) = \overline{\int_\alpha^z \frac{d\xi}{\pi(\xi +\overline{\zeta})^2}} = \frac{1}{\pi}\left(\frac{1}{\overline{\alpha}+\zeta} - \frac{1}{\overline{z}+\zeta}\right) \quad \text{and} \quad k_z^{\mathcal{D}^\alpha(\mathbb{C}_+)}(\alpha)=1.
\end{displaymath}
And thus
\begin{displaymath}
	k_z^{\mathcal{D}^\alpha(\mathbb{C}_+)}(\zeta) = \int_\alpha^\zeta \left(\frac{1}{\overline{\alpha}+\xi} - \frac{1}{\overline{z}+\xi}\right) \, \frac{d\xi}{\pi} + 1 = \ln \left(\frac{(\overline{\alpha}+\zeta)(\alpha +\overline{z})}{2\pi\re(\alpha)(\overline{z}+\zeta)}\right)+1.
\end{displaymath}
For any $\beta \in \mathbb{C}_+$, $\|\cdot\|_{\mathcal{D}^\beta(\mathbb{C}_+)}$ is an equivalent norm on $\mathcal{D}^\alpha(\mathbb{C}_+)$, i.e. for all $F \in \mathcal{D}^\alpha(\mathbb{C}_+)$,
\begin{align*}
\|F\|^2_{\mathcal{D}^\alpha(\mathbb{C}_+)} &\defeq \left\| F' \right\|^2_{\mathcal{B}_0(\mathbb{C}_+)} + |F(\alpha)|^2 \\
&\defeq \left\| F' \right\|^2_{\mathcal{B}_0(\mathbb{C}_+)} + \left|\left\langle F, \, k^{\mathcal{D}^\beta(\mathbb{C}_+)}_\alpha \right\rangle_{\mathcal{D}^\beta(\mathbb{C}_+)}\right|^2 \\
&\stackrel{\text{Cauchy-Schwarz}}{\leq} \left\| F' \right\|^2_{\mathcal{B}_0(\mathbb{C}_+)} + \left\|F\right\|^2_{\mathcal{D}^\beta(\mathbb{C}_+)} \left\|k_\alpha^{\mathcal{D}^\beta(\mathbb{C}_+)}\right\|^2_{\mathcal{D}^\beta(\mathbb{C}_+)} \\
&\leq \left(1+\left\|k_\alpha^{\mathcal{D}^\beta(\mathbb{C}_+)}\right\|^2_{\mathcal{D}^\beta(\mathbb{C}_+)}\right) \left\|F\right\|^2_{\mathcal{D}^\beta(\mathbb{C}_+)}.
\end{align*}
The set $\mathcal{D}^\alpha(\mathbb{C}_+)$ properly contains $\mathcal{D}(\mathbb{C}_+)$, since
\begin{align}
\begin{split}
	\left\|\mathfrak{L}[f]\right\|^2_{\mathcal{D}^\alpha(\mathbb{C}_+)} &\defeq	 \int_0^\infty |f(t)|^2 t \, dt + \left|\int_0^\infty f(t)e^{-\alpha t} \, dt \right|^2 \\
	&\leq \max \left\{1, \frac{1}{2\re(\alpha)} \right\}\int_0^\infty |f(t)|^2 (1+t) \, dt,
	\end{split}
	\label{eq:dirichlets}
\end{align}
and all the constant functions belong to $\mathcal{D}^\alpha(\mathbb{C}_+)$, while they cannot be in $\mathcal{D}(\mathbb{C}_+)$, because they are not in $H^2(\mathbb{C}_+)$. Moreover $\mathcal{D}(\mathbb{C}_+) \subset \mathcal{D}^\alpha(\mathbb{C}_+) \setminus \mathbb{C}$, since for example $(z+1)^{1/2} \in \mathcal{D}^\alpha(\mathbb{C}_+) \setminus \left(H^2(\mathbb{C}_+) \cup \mathbb{C} \right)$.

\begin{prop}
For all $\alpha \in \mathbb{C}_+$ we have that $CM(\mathcal{D}^\alpha(\mathbb{C}_+)) \subset CM(\mathcal{D}(\mathbb{C}_+))$, and the inclusion is proper.
\end{prop}

\begin{proof}
The inclusion $CM(\mathcal{D}^\alpha(\mathbb{C}_+)) \subseteq CM(\mathcal{D}(\mathbb{C}_+))$ is obvious by \eqref{eq:dirichlets}. It is proper, since whenever $\mu \in CM(\mathcal{D}^\alpha(\mathbb{C}_+))$, then
\begin{displaymath}
	\mu(\Omega) \leq \int_{\mathbb{C}_+} |1|^2 \, d\mu \leq C(\mu) \|1\|^2_{\mathcal{D}^\alpha(\mathbb{C}_+)} = C(\mu),
\end{displaymath}
for all $\Omega \subset \mathbb{C}_+$ and some $C(\mu)>0$, not depending on $\Omega$. That is, $\mu$ must be bounded, whereas $\delta_0 \otimes \lambda$ is clearly an unbounded measure, which belongs to $CM(\mathcal{D}(\mathbb{C}_+))$.
\end{proof}

\begin{theorem} Let $\mu$ be a positive Borel measure on $\mathbb{C}_+$.
\begin{enumerate}
	\item The measure $\mu$ is a Carleson measure for $\mathcal{D}^\alpha(\mathbb{C}_+)$ if and only if there exists a constant $C(\alpha,\, \mu)>0$ such that
	\small
	\begin{displaymath}
		\int_{\mathbb{C}_+} \left|\int_{\mathbb{C}_+} G(\zeta)\frac{\overline{\zeta-\alpha}}{(z+\overline{\alpha})(z+\overline{\zeta})} \, d\mu(\zeta) \right|^2 \, dz \leq C(\alpha, \, \mu) \int_{\mathbb{C}_+} |G|^2 \, d\mu - (\ln\pi-1)^2\left|\int_{\mathbb{C}_+} G \, d\mu \right|^2,
	\end{displaymath}
	\normalsize
	for all $G \in L^2(\mathbb{C}_+, \, \mu)$.
	\item The measure $\mu$ is a Carleson measure for $\mathcal{D}(\mathbb{C}_+)$ if and only if there exists a constant $D(\mu)>0$ such that
	\begin{displaymath}
		\int_{\mathbb{C}_+} \left|\int_{\mathbb{C}_+} \frac{G(\zeta)}{z+\overline{\zeta}} \, d\mu(\zeta) \right|^2 \, \frac{dz}{\pi e^{2\re(z)}} \leq D(\mu) \int_{\mathbb{C}_+} |G|^2 \, d\mu,
	\end{displaymath}
	for all $G \in L^2(\mathbb{C}_+, \, \mu)$.
\end{enumerate}
\end{theorem}

\begin{proof} ~
\begin{enumerate}
	\item A positive Borel measure  $\mu$ on $\mathbb{C}_+$ is a Carleson measure for $\mathcal{D}^\alpha(\mathbb{C}_+)$ if and only if the adjoint of the inclusion map $\iota^*: L^2(\mathbb{C}_+, \, \mu)~\hookrightarrow~\mathcal{D}^\alpha(\mathbb{C}_+)$ is bounded, that is there exists $C(\alpha, \, \mu)>0$ such that
	\begin{equation}
		\|\iota^* G\|^2_{\mathcal{D}^\alpha(\mathbb{C}_+)} \leq C(\alpha, \, \mu) \| G\|^2_{L^2(\mathbb{C}_+, \, \mu)},
	\label{eq:iotadirichletalpha}
	\end{equation}
	for all $G \in L^2(\mathbb{C}_+, \, \mu)$. Also
	\begin{align}
	\begin{split}
	\iota^* G(z) &\defeq \left\langle \iota^* G, \, k^{\mathcal{D}^\alpha(\mathbb{C}_+)}_z \right\rangle_{\mathcal{D}^\alpha(\mathbb{C}_+)} \defeq \left\langle G, \, k^{\mathcal{D}^\alpha(\mathbb{C}_+)}_z \right\rangle_{L^2(\mathbb{C}_+, \, \mu)} \\
	&\defeq \int_{\mathbb{C}_+} G(\zeta)\left(\ln \left(\frac{(\alpha+\overline{\zeta})(\overline{\alpha} +z)}{2\pi\re(\alpha)(z+\overline{\zeta})}\right)+1\right) \, d\mu(\zeta),
	\end{split}
	\label{eq:iotadirichletalphag}
	\end{align}
	for all $z \in \mathbb{C}_+$. And so
	\small
	\begin{displaymath}
	C(\alpha, \, \mu)\|G\|^2_{L^2(\mathbb{C}_+, \, \mu)} \stackrel{\eqref{eq:iotadirichletalpha}, \, \eqref{eq:iotadirichletalphag}}{\geq} \int_{\mathbb{C}_+} \left|\int_{\mathbb{C}_+} G(\zeta)\frac{\overline{\zeta-\alpha}}{(z+\overline{\alpha})(z+\overline{\zeta})} \, d\mu(\zeta) \right|^2 \, dz+(\ln\pi-1)^2\left|\int_{\mathbb{C}_+} G \, d\mu \right|^2,
	\end{displaymath}
	\normalsize
	as required.
	\item By the equation \eqref{eq:kernel} we know that
	\begin{displaymath}
	k^{\mathcal{D}(\mathbb{C}_+)}_z(\zeta) = \int_0^\infty \frac{e^{-t(\overline{z}+\zeta)}}{1+t} \, dt \qquad \qquad (\forall (z, \zeta) \in \mathbb{C}_+).
	\end{displaymath}
	And then, similarly as in 1. we have
	\small
	\begin{align*}
	D(\mu) \|G\|^2_{L^2(\mathbb{C}_+, \, \mu)} &\geq \left\|\left\langle G, \, k^{\mathcal{D}(\mathbb{C}_+)}_\cdot \right\rangle_{L^2(\mathbb{C}_+, \, \mu)} \right\|^2_{\mathcal{D}(\mathbb{C}_+)} \\
	&-\int_0^\infty \left|\mathfrak{L}^{-1}\left[ \int_{\mathbb{C}_+} G(\zeta) \int_0^\infty \frac{e^{-\tau(z+\overline{\zeta})}}{1+\tau} \, d\tau \, d\mu(\zeta) \right](t)\right|^2 \, (1+t) \, dt \\ 
	&=\int_0^\infty \left|\int_{\mathbb{C}_+} G(\zeta) e^{-t\overline{\zeta}} \, d\mu(\zeta) \right|^2 \frac{dt}{1+t}.
	\end{align*}
	\normalsize
	Now,
	\begin{displaymath}
		\frac{1}{1+t} = 2 \int_0^\infty e^{-2r(t+1)} \, dr = 2\pi \int_0^\infty e^{-2rt} \, \frac{dr}{\pi e^{2r}},
	\end{displaymath}
	hence by Theorem \ref{thm:mainthm}
	\begin{align*}
		D \|G\|^2_{L^2(\mathbb{C}_+, \, \mu)} &\geq \int_{\mathbb{C}_+} \left|\int_0^\infty \int_{\mathbb{C}_+} G(\zeta) e^{-t\overline{\zeta}} \, d\mu(\zeta) e^{tz} \right|^2 \frac{dz}{\pi e^{2\re(z)}} \\
		&= \int_{\mathbb{C}_+} \left| \int_{\mathbb{C}_+} \frac{G(\zeta)}{z+\overline{\zeta}} \, d\mu(\zeta) \right|^2 \, \frac{dz}{\pi e^{2\re(z)}}.
	\end{align*}
\end{enumerate}
\end{proof}

\begin{corollary}
If $\mu$ is a Carleson measure for $\mathcal{D}(\mathbb{C}_+)$, then there exists a constant $C(\mu)>0$ such that for all $a \in \mathbb{C}_+$ we have
\begin{displaymath}
	\int_{\mathbb{C}_+} \left(\frac{\mu(Q(a) \cap Q(z))}{e^{\re(z)}\re(z)}\right)^2 \, dz \leq C(\mu) \mu(Q(a)).
\end{displaymath} 
\end{corollary}

\begin{proof}
By the previous theorem, applied with $G= \chi_{Q(a)}$, we get
\begin{displaymath}
	\mu(Q(a)) \gtrapprox \int_{\mathbb{C}_+} \left| \int_{Q(a)}\frac{d\mu(\zeta)}{z+\overline{\zeta}} \right|^2 \, \frac{dz}{e^{2\re(z)}}.
\end{displaymath}
Now
\begin{equation}
	\re\left(\frac{1}{z+\overline{\zeta}}\right)= \frac{\re(z)+\re(\zeta)}{|z+\overline{\zeta}|^2} \geq 0,
	\label{eq:positiverealkernel}
\end{equation}
so for any $z \in \mathbb{C}_+$,
\begin{align*}
	\left| \int_{Q(a)}\frac{d\mu(\zeta)}{z+\overline{\zeta}} \right| &\geq \re \left(\int_{Q(a)}\frac{d\mu(\zeta)}{z+\overline{\zeta}}\right) \\
	&\stackrel{\eqref{eq:positiverealkernel}}{\geq} \int_{Q(a)\cap Q(z)} \frac{\re(z)+\re(\zeta)}{(\re(z)+\re(\zeta))^2+|\im(z)-\im(\zeta)|^2} \, d\mu(\zeta)\\
	&\stackrel{\eqref{eq:carlesonsquare}}{\geq} \int_{Q(a)\cap Q(z)} \frac{\re(z)}{10(\re(z))^2} \, d\mu(\zeta) = \frac{\mu(Q(a) \cap Q(z))}{10\re(z)},
\end{align*}
and the result follows.
\end{proof}

\section{Carleson embeddings for trees}
\label{sec:trees}

To investigate sufficient conditions for a measure to be Carleson for the Dirichlet space and similar spaces, let us now turn our attention to trees. This approach was introduced in \cite{arcozzi2002} to classify Carleson measures for analytic Besov spaces on the unit disk of the complex plane and in \cite{arcozzi2008} for Drury-Averson Hardy space and Besov-Sobolev spaces on complex $n$-balls.

Consider a tree $T$ with a partial order relation "$\leq$" defined on the set of its vertices. We will write $v \in T$ to denote that $v$ is a vertex of $T$, and in general associate $T$ with the set of its vertices only. Let $x, \, y$ be two distinct elements (vertices) of $T$. If for all $c \in T$ such that $x \leq c \leq y$ we have $x=c$ or $y=c$, then we call $y$ the \emph{predecessor of $x$} and write $y:= x^-$. 
For any $\varphi : T \longrightarrow \mathbb{C}$ we define the \emph{primitive $\mathcal{I}$ of $\varphi$ at $x \in T$} to be
\begin{displaymath}
	\mathcal{I}\varphi(x) := \sum_{y \leq x} \varphi(y).
\end{displaymath}
And finally, we let
\begin{displaymath}
	S(x) := \left\{y \in T \; : \; y \geq x\right\} \quad \text{ and } \quad S(-\infty) := T.
\end{displaymath}
The following two lemmata appear in \cite{arcozzi2002} in a similar form. The first of them is, however,  only stated for rooted trees. This would cause a problem in Lemma \ref{secondarcozzilemma}, because if we decide to decompose a half-plane into subsets, each of them corresponding to a vertex of some tree, and we let one of this sets correspond to the root of the tree, then we would only restrict our consideration of Carleson measures to those which are bounded on $\mathbb{C}_+$. In order to avoid this, we shall rephrase the statement of Lemma \ref{firstarcozzilemma} (part of Theorem 3 in \cite{arcozzi2002}, p. 447) in order to incorporate rootless trees as well, and amend the proof where necessary.

\begin{lemma}
\label{firstarcozzilemma}
Let $T$ be a tree with a partial order defined on it, let $1<p \leq q < \infty$, and let $p'=p/(p-1)$, $q':=q/(q-1)$ be the adjoint indices of $p$ and $q$. Let also $\rho$ be a weight on $T$, and $\mu$ be a non-negative function on $T$. If there exists a constant $C(\mu, \, \rho)>0$ such that for all $r \in T \cup \{-\infty\}$,
\begin{equation}
\left(\sum_{x \in S(r)} \left(\sum_{y \in S(x)} \mu(y) \right)^{p'} \rho(x)^{1-p'} \right)^{q'/p'} \leq C(\mu, \rho) \sum_{x \in S(r)} \mu(x),
\label{eq:0}
\end{equation}
then there exists a constant $C'(\mu, \, \rho)>0$ such that for all $\varphi : T \longrightarrow \mathbb{C}$,
\begin{displaymath}
\left(\sum_{x \in T} |\mathcal{I}\varphi(x)|^q \mu(x) \right)^{1/q} \leq C'(\mu, \, \rho) \left(\sum_{x \in T} |\varphi(x)|^p \rho(x) \right)^{1/p}.
\end{displaymath}
\end{lemma}

\begin{proof}
Let $g \in L^p(T, \, \rho)$. To prove this lemma we only need to show that
\begin{displaymath}
\|\mathcal{I}g\|_{L^q(T, \,  \mu)} \leq C'(\mu, \, \rho)\|g\|_{L^p(T, \, \rho)},
\end{displaymath}
for all $g \geq 0$, in which case $\mathcal{I}g$ is non-decreasing with respect to the order relation on $T$. Let
\begin{displaymath}
	\Omega_k := \left\{x \in T \; : \; \mathcal{I}g(x) > 2^k \right\} = \bigcup_j S(r_j^k),
\end{displaymath}
where $\{r_j^k \in T \; : \; j=1, \, \ldots \, \}$ is the set of minimal points in $\Omega_k$ with respect to the partial order on $T$, if such points exist. Otherwise we define $r_1^k := - \infty$ and
\begin{displaymath}
	\Omega_k := \left\{x \in T \; : \; \mathcal{I}g > 2^k \right\} \defeq S(r_1^k) \defeq S(-\infty) \defeq T.
\end{displaymath}
Let $E_j^k = S(r_j^k) \cap \left(\Omega_{k+1} \setminus \Omega_{k+2} \right)$. Then for $x \in E_j^k$ we get
\begin{displaymath}
\mathcal{I}(\chi_{S(r_j^k)}g)(x) = \sum_{r_j^k \leq y \leq x} g(y) = \mathcal{I}g(x) - \mathcal{I}g((r_j^k)^-) \geq 2^k,
\end{displaymath}
where we adopt a convention that $\mathcal{I}g((r_j^k)^-) :=0$, whenever $r_j^k = - \infty$. Let $\tilde{\mu}$ be a measure on the $\sigma$-algebra $\mathcal{P}(T)$ (the power set of $T$) defined by $\tilde{\mu}(\{x\}) := \mu(x)$, for all $x \in T$. Thus we have,
\begin{align*}
2^k\tilde{\mu}(E^k_j) &= 2^k \sum_{x \in E^k_j} \mu(x) \\
&\leq \sum_{x \in E^k_j} \mathcal{I}(\chi_{S(r_j^k)}g)(x) \mu(x) \\
&=\sum_{y \in S(r_j^k)}g(y) \sum_{x \in E_j^k, \, x \geq y} \mu(x) \\
&=\sum_{y \in S(r_j^k)}g(y) \mathcal{I}^* \chi_{E^j_k}(y) \\
&=\sum_{y \in S(r_j^k) \cap (\Omega_{k+2} \cup \Omega_{k+2}^c)}g(y) \mathcal{I}^* \chi_{E^k_j}(y) \\
&=\sum_{y \in S(r_j^k) \cap \Omega_{k+2}} g(y) \mathcal{I}^* \chi_{E^k_j}(y) + \sum_{y \in S(r_j^k) \cap \Omega_{k+2}^c} g(y) \mathcal{I}^* \chi_{E^k_j}(y),
\end{align*}
where $\Omega^c_{k+2}$ denotes the complement of $\Omega_{k+2}$ in $\mathbb{C}_+$. But since $\mathcal{I}^* \chi_{E^k_j}(y)=0$ for all $y \in \Omega_{k+2}$,
\begin{equation}
	2^k\tilde{\mu}(E^k_j) \leq \sum_{y \in S(r_j^k) \cap \Omega_{k+2}^c} g(y)\mathcal{I}^* \chi_{E^k_j}(y).
\label{eq:secondsumest}
\end{equation}
Now,
\begin{align*}
\sum_{x \in T} |\mathcal{I}g(x)|^q \mu(x) &\leq \sum_{k \in \mathbb{Z}} \tilde{\mu}\left\{x \in T \; : \; 2^{k+1} < \mathcal{I}g(x) \leq 2^{k+2}\right\} 2^{(k+2)q}\\
&= 2^{2q} \sum_{k \in \mathbb{Z}} 2^{kq}\tilde{\mu}\left(\Omega_{k+1} \setminus \Omega_{k+2}\right) \\
&\leq 2^{2q} \sum_{k \in \mathbb{Z}, \, j} \tilde{\mu}(E_j^k) 2^{kq} \\
&= 2^{2q} \left(\sum_{(k,j) \in E} \tilde{\mu}(E_j^k) 2^{kq} + \sum_{(k,j) \in F} \tilde{\mu}(E_j^k) 2^{kq}\right),
\end{align*}
where
\begin{align}
\label{eq:E}
E &:= \left\{(k,j) \; : \; \tilde{\mu}(E^k_j) \leq \beta \tilde{\mu}(S(r^j_k)) \right\}, \\
\label{eq:F}
F &:= \left\{(k,j) \; : \; \tilde{\mu}(E^k_j) > \beta \tilde{\mu}(S(r^j_k)) \right\},
\end{align}
for some $0<\beta < \frac{1-2^{-q}}{2}$. Let $\left\{x^n_k \right\}_{k,n} \subseteq T \cup \{\varnothing\}$ be a collection of distinct elements of this set, such that $\left\{x^k_n \right\}_{k,n} = \Omega_k \setminus \Omega_{k+1}$, for all $k \in \mathbb{Z} \setminus \{0\}$. Then
\begin{align*}
\sum_{(k,j), k \geq 1} \tilde{\mu}(S(r^k_j))2^{kq} &= \sum_{k=1}^\infty \tilde{\mu}(\Omega_k)2^{kq} = \sum_{k=1}^\infty \tilde{\mu}(\Omega_k \setminus \Omega_{k+1})\sum_{l=1}^k 2^{lq} \\
&= \sum_{k=0}^\infty \tilde{\mu}(\{x^k_n\}_n) \sum_{l=0}^{k-1} 2^{(k-l)q} \leq \sum_{k=0}^\infty \sum_n \mu(x^k_n) |\mathcal{I}g(x^k_n)|^q \sum_{l=0}^{k-1}2^{-lq} \\
&\leq \frac{1}{1-2^{-q}} \sum_{x \in T} |\mathcal{I}g(x)|^q \mu(x) \defeq \frac{1}{1-2^{-q}} \|\mathcal{I}g\|^q_{L^q(T, \, \mu)}.
\end{align*}
Similarly,
\begin{align*}
\sum_{(k,j), k < 1} \tilde{\mu}(S(r^j_k))2^{kq} &= \sum_{k=-\infty}^0 \tilde{\mu}(\Omega_k)2^{kq} \\
&= \left(\sum_{l=0}^\infty 2^{-lq} \right)\left(\tilde{\mu}(\Omega_0) + \sum_{k=1}^\infty \tilde{\mu}(\Omega_{-k} \setminus \Omega_{-k+1}) 2^{-kq} \right)\\
&\leq \frac{1}{1-2^{-q}} \left(\sum_j \mu(r^0_j) |\mathcal{I}g(r^0_j)|^q + \sum_{k=1}^\infty \mu(x^{-k}_n)|\mathcal{I}g(x^{-k}_n)|^q \right) \\
&\leq \frac{1}{1-2^{-q}} \sum_{x \in T} |\mathcal{I}g(x)|^q \mu(x) \defeq \frac{1}{1-2^{-q}} \|\mathcal{I}g\|^q_{L^q(T, \, \mu)}.
\end{align*}
So
\begin{displaymath}
	2^{2q} \sum_{(k,j) \in E} \tilde{\mu}(E^k_j) 2^{kq} \stackrel{\eqref{eq:E}}{\leq}  \frac{2^{2q+1}}{1-2^{-q}} \beta \|\mathcal{I}g\|^q_{L^q(T, \, \mu)}.
\end{displaymath}
For the sum indexed by $F$ we have
\begin{align*}
\sum_{(k,j) \in F} &\tilde{\mu}(E_j^k) 2^{kq} \stackrel{\eqref{eq:secondsumest}}{\leq} \sum_{(k,j) \in F} \tilde{\mu}(E^k_j) \left|\tilde{\mu}(E^k_j)^{-1} \sum_{y \in S(r_j^k) \cap \Omega_{k+2}^c} g(y)\mathcal{I}^* \chi_{E^k_j}(y) \right|^q \\
&\stackrel{\eqref{eq:F}}{\leq} \beta^{1-q} \sum_{(k,j) \in F} \tilde{\mu}(S(r^k_j))^{1-q} \left| \sum_{y \in S(r_j^k) \cap \Omega_{k+2}^c} g(y)\mathcal{I}^* \chi_{E^k_j}(y) \right|^q \\
&\stackrel{\text{H\"{o}lder's}}{\leq} \beta^{1-q} \sum_{(k,j) \in F} \tilde{\mu}(S(r^k_j))^{1-q} \\
&\times \left(\sum_{y \in S(r_j^k) \cap \Omega_{k+2}^c} \left|\mathcal{I}^* \chi_{E^k_j}(y)\right|^{p'} \rho(y)^{1-p'} \right)^{\frac{q}{p'}} \left(\sum_{y \in S(r_j^k) \cap \Omega_{k+2}^c} \left|g(y)\right|^{p} \rho(y) \right)^{\frac{q}{p}}\\
\end{align*}
\begin{align*}
&\leq \beta^{1-q} \sum_{(k,j) \in F} \left(\sum_{x \in S(r^k_j)} \mu(x) \right)^{1-q} \\
&\times \left(\sum_{x \in S(r^k_j)} \left(\sum_{y \in S(x)} \mu(y) \rho(y)^{1-p'} \right)^{p'} \right)^{\frac{q'(q-1)}{p'}} \left(\sum_{y \in S(r_j^k) \cap \Omega_{k+2}^c} \left|g(y)\right|^{p} \rho(y) \right)^\frac{q}{p} \\
&\stackrel{\eqref{eq:0}}{\leq} \beta^{-q} C^{q-1} \sum_{(k,j)}  \left(\sum_{y \in S(r_j^k) \cap \Omega_{k+2}^c} \left|g(y)\right|^{p} \rho(y) \right)^{q/p} \\
&\stackrel{q\geq p}{\leq} \beta^{-q} C^{q-1} \left(\sum_{(k,j)} \sum_{y \in S(r_j^k) \cap \Omega_{k+2}^c} \left|g(y)\right|^{p} \rho(y) \right)^{q/p} \\
&= C^{q-1} \beta^{-q} \left( \sum_{k \in \mathbb{Z}} \sum_{y \in \Omega_k \cap \Omega_{k+2}^c} \left|g(y)\right|^{p} \rho(y) \right)^{q/p} \\
&= C^{q-1} \beta^{-q} \left( \sum_{k \in \mathbb{Z}} \sum_{y \in \Omega_k \setminus \Omega_{k+2}} \left|g(y)\right|^{p } \rho(y) \right)^{q/p} \\
&= 2^{q/p} \beta^{-q} C^{q-1} \left(\sum_{x \in T} \left|g(x)\right|^{p}  \rho(y) \right)^{q/p} \\
&\defeq 2^{q/p} C^{q-1} \|g\|^q_{L^p(T, \, \rho)}.
\end{align*}
Therefore we can conclude that
\begin{displaymath}
 \|\mathcal{I}g\|^q_{L^q(T, \, \mu)} \leq \frac{2^{2q+1}}{1-2^{-q}} \beta \|\mathcal{I}g\|^q_{L^q(\mu)} + 2^{q/p} C^{q-1} \beta^{1-q}\|g\|^q_{L^p(T(\zeta))},
\end{displaymath}
and since
\begin{displaymath}
	\beta < \frac{1-2^{-q}}{2^{2q+1}},
\end{displaymath}
we get the desired result.
\end{proof}

Given $\zeta \in \mathbb{C}_+$, consider the following decomposition of the right complex half-plane: for any $(k, l) \in \mathbb{Z}^2$ let
\begin{displaymath}
	R_{(k,l)}(\zeta) := \left\{z \in \mathbb{C}_+ \; : \; 2^{k-1} < \frac{\re(z)}{\re(\zeta)} \leq 2^{k}, \, 2^k l \leq \frac{\im(z)-\im(\zeta)}{\re{\zeta}} < 2^k (l+1)\right\}.
\end{displaymath}

\setlength{\unitlength}{12mm}
\begin{picture}(20,10.5)(-0.5,-5.25)
\thicklines
\put(0,-5.25){\vector(0,1){10.5}}
\put(-1,5){$\im(z)$}

\put(1,0){\circle*{0.075}}
\put(0.75,-0.25){\footnotesize{$\zeta$}} 

\thinlines

\put(8,-5.2){\line(0,1){10.4}}
\put(4,0){\line(1,0){6}} 
\put(5.5,-4){\large{$R_{(3,-1)}(\zeta)$}}
\put(5.5,4){\large{ $R_{(3,0)}(\zeta)$}}

\put(4,-5.2){\line(0,1){10.4}}
\multiput(2,-4)(0,4){3}{\line(1,0){2}} 
\put(2.5,-2){$R_{(2,-1)}(\zeta)$}
\put(2.5,2){ $R_{(2,0)}(\zeta)$}

\put(2,-5.2){\line(0,1){10.4}}
\multiput(1,-4)(0,2){5}{\line(1,0){1}} 
\put(1,-5){\tiny{$R_{(1,-3)}(\zeta)$}}
\put(1,-3){\tiny{$R_{(1,-2)}(\zeta)$}}
\put(1,-1){\tiny{$R_{(1,-1)}(\zeta)$}}
\put(1,1){\tiny{ $R_{(1,0)}(\zeta)$}}
\put(1,3){\tiny{ $R_{(1,1)}(\zeta)$}}
\put(1,5){\tiny{ $R_{(1,2)}(\zeta)$}}

\put(1,-5.2){\line(0,1){10.4}}
\multiput(0.5,-5)(0,1){11}{\line(1,0){0.5}} 

\put(0.5,-5.2){\line(0,1){10.4}}
\multiput(0.25,-5)(0,0.5){21}{\line(1,0){0.25}} 

\put(0.25,-5.2){\line(0,1){10.4}}
\multiput(0.125,-5)(0,0.25){41}{\line(1,0){0.125}} 

\put(0.125,-5.2){\line(0,1){10.4}}
\multiput(0.0625,-5.125)(0,0.125){82}{\line(1,0){0.0625}} 

\put(0.0625,-5.2){\line(0,1){10.25}}
\multiput(0.03125,-5.1875)(0,0.0625){164}{\line(1,0){0.03125}} 

\put(0.03125,-5.2){\line(0,1){10.25}}
\multiput(0.015625,-5.1875)(0,0.03125){327}{\line(1,0){0.015625}} 

\end{picture}

We can view each element of the set of rectangles $\{R_{(k,l)}(\zeta) : (k, \, l) \in \mathbb{Z}^2 \}$ as a vertex of an abstract graph. If we also have that $x, \, y \in \{R_{(k,l)}(\zeta) : (k, \, l) \in \mathbb{Z}^2 \}$ and $\overline{x} \cap \overline{y}$ is a vertical segment in $\mathbb{C}_+$, then we say there is an edge between $x$ and $y$.  With this convention, these vertices and edges form an abstract tree, which we shall denote by $T(\zeta)$. Let $A(\cdot)$ a positive function on the set vertices of $T(\zeta)$ assigning to each of them the area of the corresponding rectangle from $\{R_{(k,l)}(\zeta):(k, \, l) \in \mathbb{Z}^2 \}$. We can define a partial order on $T(\zeta)$ by considering the unique path between each pair $x, \, y \in T(\zeta)$. If for each vertex $c$ lying on this path, $A(x) \geq A(c) \geq A(y)$, then $x \leq y$. With this setting and the following definition, we may proceed to prove next lemma, which has a disk counterpart in \cite{arcozzi2002} (part of Theorem 1, p. 445) using the Whitney decomposition of $\mathbb{D}$.

\begin{definition}
A positive weight $\rho : \mathbb{C}_+ \longrightarrow (0, \, \infty)$ is called \emph{regular} if for all $\varepsilon>0$ there exists $\delta >0$ such that $\rho(z_1) \leq \delta \rho(z_2)$, whenever $z_1$ and $z_2$ are within (Poincar\'{e}) hyperbolic right half-plane distance $\varepsilon$, i.e.
\begin{displaymath}
	d_H (z_1, \, z_2) \defeq \cosh^{-1} \left(1+ \frac{(\re(z_1)-\re(z_2))^2+(\im(z_1)-\im(z_2))^2}{2\re(z_1)\re(z_2)}\right) \leq \varepsilon.
\end{displaymath}
\end{definition}

\begin{lemma}
\label{secondarcozzilemma}
Let $\rho : \mathbb{C}_+ \longrightarrow (0, \, \infty)$ be regular, let $\mu$ be a positive Borel measure on $\mathbb{C}_+$. If there exists a constant $C(\mu, \rho)>0$, such that for all $a \in \mathbb{C}_+$ we have
\begin{equation}
\left(\int_{Q(a)} \frac{(\mu(Q(a) \cap Q(z))^{p'}}{(\re(z))^2} \rho(z)^{1-p'} \, dz \right)^{q'/p'} \leq C(\mu, \rho) \mu(Q(a)),
\label{eq:1}
\end{equation}
then there exists a constant $C'(\mu, \rho)>0$ such that
\begin{displaymath}
\left(\sum_{\beta \geq \alpha} \left(\sum_{\gamma \geq \beta} \mu(\gamma)\right)^{p'} \tilde{\rho}(\beta)^{1-p'} \right)^{q'/p'} \leq C' \sum_{\beta \geq \alpha} \mu(\beta),
\end{displaymath}
for all $\alpha \in T(\zeta)$. Here $\tilde{\rho}(\beta)$ is defined to be $\rho(z_\beta)$, for some fixed $z_\beta \in \beta \subset \mathbb{C}_+$, for all $\beta \in T(\zeta)$.
\end{lemma}

\begin{proof}
Choose any $\zeta \in \mathbb{C}_+$. Then for all $\alpha \in T(\zeta)$ there exists $a \in \mathbb{C}_+$ such that
\begin{equation}
	Q(a) = \bigcup_{\beta \geq \alpha} \beta \qquad \text{ and } \qquad \mu(Q(a)) = \sum_{\beta \geq \alpha}  \mu(\beta)
\label{eq:2}
\end{equation}
(or to be precise: this holds after removing some horizontal lines from some the sets $\beta \geq \alpha$, to avoid covering the same set twice, and otherwise keeping the tree model intact). Given $\beta \geq \alpha$, let $(k, \, l) \in \mathbb{Z}^2$ be such that $\beta = R_{(k, \, l)}(\zeta)$ and let
\footnotesize
\begin{displaymath}
S(\beta) := \left\{z \in \mathbb{C}_+ \, : \, 2^{k-1} < \frac{\re(z)}{\re(\zeta)} \leq 2^k, \, \frac{\left|\im(z)-\im(\zeta)-2^k\left(l+\frac{1}{2}\right)\re(\zeta)\right|}{\tan\left(\frac{\pi}{4}\right)} < \re(z)-2^{k-1}\right\},
\end{displaymath}
\normalsize
Now
\begin{displaymath}
	\bigcup_{\gamma \geq \beta} \gamma \subseteq Q(z),
\end{displaymath}
whenever $z \in S(\beta) \subset \beta \geq \alpha$, and also
\begin{equation}
 Q(a) \cap Q(z) \supseteq \bigcup_{\gamma \geq \beta} \gamma.
\label{eq:3}
\end{equation}
We also have that for any $z_1$ and $z_2$ in $\beta$
\begin{displaymath}
	d_H (z_1, \, z_2) \leq \cosh^{-1}\left(1+ \frac{2^{2k-2}+2^{2k}}{2^{2k-2}}\right) = \cosh^{-1} \left(\frac{7}{2}\right),
\end{displaymath}
which does not depend on the choice of $\beta \in T(\zeta)$, so there exists $\delta > 0$ such that
\begin{align*}
C\sum_{\beta \geq \alpha} \mu(\beta) &\stackrel{\eqref{eq:2}}{=} C\mu(Q(a)) \stackrel{\eqref{eq:1}}{\geq} \left(\int_{Q(a)} \frac{(\mu(Q(a)) \cap Q(z))^{p'}}{(\re(z))^2} \rho(z)^{1-p'}\, dz \right)^{q'/p'} \\
&\stackrel{\eqref{eq:2}}{\geq} \delta^{q'/p'}\left( \sum_{\beta \geq \alpha} \rho(\beta)^{1-p'} \int_\beta \frac{\left(\mu(Q(a) \cap Q(z))\right)^{p'}}{(\re(z))^2} \, dz \right)^{q'/p'} \\
&\geq \delta^{q'/p'} \left( \sum_{\beta \geq \alpha} \rho(\beta)^{1-p'} \int_{S(\beta)} \frac{\left(\mu(Q(a) \cap Q(z))\right)^{p'}}{(\re(z))^2} \, dz \right)^{q'/p'}\\
&\stackrel{\eqref{eq:3}}{\geq} \delta^{q'/p'} \left( \sum_{\beta \geq \alpha} \rho(\beta)^{1-p'} (\mu( \bigcup_{\gamma \geq \beta} \gamma))^{p'} \right)^{q'/p'} \\
&= \delta^{q'/p'} \left( \sum_{\beta \geq \alpha} \left(\sum_{\gamma \geq \beta} \mu(\gamma) \right)^{p'} \rho(\beta)^{1-p'} \right)^{q'/p'},
\end{align*}
as required.
\end{proof}

The following theorem is a half-plane and Hilbertian version of Theorem 1 from \cite{arcozzi2002}.

\begin{theorem}
\label{thm:arcozzi}
Let $\rho$ be a regular weight such that $\|F\|^2_{A^2_{(m)}} \geq \int_{\mathbb{C}_+} |F'(z)|^2 \rho(z) \, dz$, for all $F \in A^2_{(m)}$, and let $\mu$ be a positive Borel measure on $\mathbb{C}_+$. If
\begin{equation}
\int_{Q(a)} \left(\frac{\mu(Q(a) \cap Q(z))}{\re(z)}\right)^2 \, \frac{dz}{\rho(z)} \leq C(\mu, \rho) \mu(Q(a)),
\label{eq:1'}
\end{equation}
for all $a \in \mathbb{C}_+$, then $\mu$ is a Carleson measure for $A^2_{(m)}$.
\end{theorem}

\begin{proof}
Let $\zeta \in \mathbb{C}_+$. Given $F \in A^2_{(m)}$, for each $\alpha \in T(\zeta)$ let $w_\alpha, \, z_\alpha \in \overline{\alpha} \subset \mathbb{C}_+$ be such that
\begin{displaymath}
	z_\alpha := \sup_{z \in \alpha} \{|F(z)|\} \qquad \text{ and } \qquad w_\alpha := \sup_{w \in \alpha} \{|F'(w)|\}.
\end{displaymath}
Define a weight $\tilde{\rho}$ on $T(\zeta)$ by $\tilde{\rho}(\alpha) := \rho(z_\alpha)$. And also: $r_\alpha = \re(w_\alpha)/4$, $\Phi(\alpha)~:=~F(z_\alpha)$, $\varphi(\alpha) = \Phi(\alpha)-\Phi(\alpha^-)$, for all $\alpha \in T(\zeta)$. Note that $\mathcal{I}\varphi = \Phi$, because
\begin{align*}
	\lim_{\alpha \longrightarrow - \infty} |F(z_\alpha)| &= \lim_{\alpha \longrightarrow - \infty} \left|\left\langle F, \, k^{A^2_{(m)}}_z \right\rangle_{A^2_{(m)}}\right| \\
	&\stackrel{\text{Cauchy-Schwarz}}{\leq} \|F\|_{A^2_{(m)}} \lim_{\alpha \longrightarrow - \infty} \int_0^\infty \frac{e^{-2t\re(z_\alpha)}}{w_{(m)}(t)} \, dt = 0.
\end{align*}
Since \eqref{eq:1'} holds, we can apply Lemma \ref{firstarcozzilemma} to $\varphi, \, \tilde{\rho}, \, \tilde{\mu}$ (where $\tilde{\mu}(\alpha):= \mu(\alpha)$, for all $\alpha \in T(\zeta)$) in the following way
\begin{align*}
\int_{\mathbb{C}_+} |F|^2 \, d\mu & = \sum_{\alpha \in T(\zeta)} \int_\alpha |F|^2 \, d\mu \leq \sum_{\alpha \in T(\zeta)} |\Phi(\alpha)|^2 \tilde{\mu}(\alpha) \\
&\stackrel{\text{Lemma }\ref{firstarcozzilemma}}{\leq} \sum_{\alpha \in T(\zeta)} |\varphi(\alpha)|^2 \tilde{\rho}(\alpha) \defeq \sum_{\alpha \in T(\zeta)} |\Phi(\alpha)-\Phi(\alpha^-)|^2 \tilde{\rho}(\alpha) \\
&\stackrel{\stackrel{\text{Fundamental Thm}}{\text{of Calculus}}}{\leq} \sum_{\alpha \in T(\zeta)} \left| \int_{z_{\alpha^-}}^{z_\alpha} F'(w) \, dw \right|^2 \tilde{\rho}(\alpha) \\
&\lessapprox \sum_{\alpha \in T(\zeta)} \text{diam}(\alpha)^2 |F'(w_\alpha)+F'(w_{\alpha^-})|^2 \tilde{\rho}(\alpha) \\
&\lessapprox \sum_{\alpha \in T(\zeta)} \text{diam}(\alpha)^2 |F'(w_\alpha)|^2 \tilde{\rho}(\alpha) \\
&\stackrel{\stackrel{\text{Mean-Value}}{\text{Property}}}{\leq} \sum_{\alpha \in T(\zeta)} \text{diam}(\alpha)^2 \left|\frac{1}{\pi r^2_\alpha} \int_{B(w_\alpha, \, r_\alpha)} F'(z) \, dz \right|^2 \tilde{\rho}(\alpha) \\
&\stackrel{\text{H\"{o}lder's}}{\leq} \sum_{\alpha \in T(\zeta)} \frac{\text{diam}(\alpha)^2}{\pi r_\alpha^2} \int_{B(w_\alpha, \, r_\alpha)} |F'(z)|^2 \, dz \tilde{\rho}(\alpha) \\
&\lessapprox \sum_{\alpha \in T(\zeta)} \int_{\bigcup_{\beta \in T(\zeta) \; : \; \beta \cap B(w_\alpha, \, r_\alpha) \neq \varnothing}} |F'(z)|^2 \rho(z) \, dz \\
&\lessapprox \sum_{\alpha \in T(\zeta)} \int_{\alpha} |F'(z)|^2 \rho(z) \, dz \\
&\leq \|F\|^2_{A^2_{(m)}}.
\end{align*}
\end{proof}

\begin{corollary}
\label{corollary2}
Let $\mu$ be a positive Borel measure on $\mathbb{C}_+$. If there exists a constant $C(\mu)>0$ such that
\begin{displaymath}
\int_{Q(a)} \left(\frac{\mu(Q(a) \cap Q(z))}{\re(z)}\right)^2 \, dz \leq C(\mu) \mu(Q(a)),
\end{displaymath}
for all $a \in \mathbb{C}_+$, then $\mu$ is a Carleson measure for $\mathcal{D}(\mathbb{C}_+)$.
\end{corollary}

Note that Theorem \ref{thm:arcozzi} cannot be applied to $\mathcal{D}^\alpha(\mathbb{C}_+)$, since the limit of its functions, as the real part of their arguments approaches infinity, is not necessarily 0.

\section{An application}
\label{sec:application}

Let $A$ be an infinitesimal generator of a $C_0$-semigroup $(T(t))_{t\geq 0}$ on a Hilbert space $\mathcal{H}$. Consider the linear system
\begin{displaymath}
\frac{dx(t)}{dt} = Ax(t)+Bu(t), \qquad x(0)=x_0, \qquad t \geq 0.
\end{displaymath}
Here $u(t) \in \mathbb{C}$ is the \emph{input} at time $t$, and $B : \mathbb{C} \longrightarrow D(A^*)'$, the \emph{control operator}, where $D(A^*)'$ denotes the completion of $\mathcal{H}$ with respect to the norm
\begin{displaymath}
	\|x\|_{D(A^*)'}:= \left\|(\beta-A)^{-1}x\right\|_{\mathcal{H}},
\end{displaymath}
for any $\beta \in \rho(A)$. To ensure that the \emph{state} $x(t)$ is in $\mathcal{H}$, we need $B~\in~\mathcal{L}(\mathbb{C}, \, D(A^*)')$ and
\begin{displaymath}
	\left\|\int_0^\infty T(t)Bu(t) \, dt \right\|_{\mathcal{H}} \leq m_0 \|u\|_{L^2_w(0, \infty)},
\end{displaymath}
for some $m_0 \geq 0$. Then we say that the control operator $B$ is $L^2_w(0, \infty)$-admissible. We refer to the survey \cite{jacob2004} and the book \cite{tucsnak2009} for the basic background to admissibility in the context of well-posed systems. The following theorem appears in \cite{jacob2013} and \cite{jacob2014} (with weaker results appearing earlier in \cite{ho1983} and \cite{weiss1988} for $H^2(\mathbb{C}_+)=\mathfrak{L}\left[L^2(0, \, \infty)\right]$, and \cite{wynn2010} for $\mathcal{B}^2_{-\alpha}(\mathbb{C}_+)=\mathfrak{L}\left[L^2_{t^\alpha}(0, \, \infty)\right]$, $-1<\alpha<0$).

\begin{theorem}
Suppose the semigroup $(T(t))_{t\geq 0}$ acts on a Hilbert space $X$ with a Riesz basis of eigenvectors $(\phi_k)$; that is, $T(t)\phi_k = e^{\lambda_k t}\phi_k$, for each $k$, $(\lambda_k)$ are the eigenvalues of eigenvectors forming a Riesz basis of $A$, each of which lies in the open left complex half-plane $\mathbb{C}_-$, and $(\phi_k)$ is a Schauder basis of $X$ such that for some constants $c, \, C >0$ we have
\begin{displaymath}
	c\sum|a_k|^2 \leq \left\|\sum a_k\phi_k\right\|^2 \leq C \sum|a_k|^2,
\end{displaymath}
for all sequences $(a_k) \in \ell^2$. Suppose also that $B$ is a linear bounded map from $\mathbb{C}$ to $D(A^*)'$ corresponding to the sequence $(b_k)$. Then the control operator $B$ is $L^2_w(0, \infty)$-admissible for $(T(t))_{t \geq 0}$ if and only if
\begin{displaymath}
	\mu:= \sum_k |b_k|^2 \delta_{-\lambda_k}
\end{displaymath}
is a Carleson measure for $\mathfrak{L}\left[L^2_w(0, \infty)\right]$.
\end{theorem}

This theorem can also be stated for observation operators (see again \cite{jacob2014}). Combining the above statement with Theorem \ref{thm:mainthm}, we get a direct application of the results established throughout this paper (i.e. if $w$ is of the form \eqref{eq:wm}, then $B$ is $L^2_w(0, \, \infty)$-admissible if and only if $\mu$ is a Carleson measure for $A^2_{(m)}$). We leave the details for the reader.

\textbf{Acknowledgment} 	The author of this article would like to thank the UK Engineering and Physical Research Council (EPSRC) and the School of Mathematics at the University of Leeds for their financial support. He is also extremely thankful to Professor Jonathan R. Partington for all the important comments and help in preparation of this research paper.

\end{document}